\documentclass[12 pt]{amsart}

\usepackage{hyperref}
\usepackage{etex}
\usepackage[shortlabels]{enumitem}
\usepackage{amsmath}
\usepackage{amsxtra}
\usepackage{amscd}
\usepackage{amsthm}
\usepackage{adjustbox}
\usepackage{amsfonts}
\usepackage{amssymb}
\usepackage{eucal}
\usepackage[all]{xy}
\usepackage{graphicx}
\usepackage{tikz-cd}
\usepackage{mathrsfs}
\usepackage{subfiles}
\usepackage{mathpazo}
\usepackage[colorinlistoftodos, textsize=tiny]{todonotes}
\usepackage{morefloats}
\usepackage{pdfpages}
\usepackage{thm-restate}
\usepackage[percent]{overpic}
\usepackage[utf8]{inputenc}
\usepackage{epigraph}
\usepackage{csquotes}
\usepackage[margin=1in]{geometry}
\usepackage{adjustbox}
\usepackage{microtype}
\usepackage{stmaryrd}
\usepackage{tikz}
\usetikzlibrary{shapes,arrows.meta}
\usepackage{tikz-3dplot} 
\usetikzlibrary{fadings}
\usetikzlibrary{decorations.pathmorphing, decorations.markings}
\usepackage{appendix}

\usepackage{verbatim}
\usepackage{stmaryrd}
\usepackage{scalerel}
\usepackage{stackengine}
\stackMath
\newcommand\reallywidehat[1]{%
\savestack{\tmpbox}{\stretchto{%
  \scaleto{%
    \scalerel*[\widthof{\ensuremath{#1}}]{\kern-.6pt\bigwedge\kern-.6pt}%
    {\rule[-\textheight/2]{1ex}{\textheight}}
  }{\textheight}%
}{0.5ex}}%
\stackon[1pt]{#1}{\tmpbox}%
}
\usepackage{mathtools}

\graphicspath{ {images/} }

\RequirePackage{color}
\definecolor{myred}{rgb}{0.75,0,0}
\definecolor{mygreen}{rgb}{0,0.5,0}
\definecolor{myblue}{rgb}{0,0,0.65}

\usepackage{color}

\usepackage{hyperref}
\hypersetup{citecolor=blue}
\usepackage{tikz}
\usetikzlibrary{matrix,arrows,decorations.pathmorphing}

\theoremstyle{plain}
\newtheorem{theorem}[subsection]{Theorem}

\newtheorem{proposition}[subsection]{Proposition}
\newtheorem{proposition-definition}[subsection]{Proposition-Definition}
\newtheorem{lemma}[subsection]{Lemma}
\newtheorem{corollary}[subsection]{Corollary}

\theoremstyle{definition}
\newtheorem{definition}[subsection]{Definition}
\newtheorem{remark}[subsection]{Remark}

\newtheorem{question}[subsection]{Question}
\newtheorem{conjecture}[subsection]{Conjecture}

\theoremstyle{remark}

\numberwithin{equation}{section}
\newcommand\nc{\newcommand}
\nc\on{\operatorname}
\nc\renc{\renewcommand}
\DeclareMathOperator\rk{rk}

\DeclareMathOperator\id{id}
\DeclareMathOperator\Hom{Hom}
\DeclareMathOperator\Gr{Gr}

\DeclareMathOperator\GL{GL}
\DeclareMathOperator\SL{SL}
\DeclareMathOperator\PGL{PGL}

\DeclareMathOperator{\Pic}{Pic}

\DeclareMathOperator{\LS}{LS}
\DeclareMathOperator{\VMHS}{VMHS}
\DeclareMathOperator{\MHS}{MHS}

\DeclareMathOperator{\Gal}{Gal}

\DeclareMathOperator{\PSL}{PSL}
\DeclareMathOperator{\VHS}{VHS}

\nc\mf\mathfrak
\nc\mc\mathcal
\nc\mb\mathbb
\nc\msf\mathsf
\nc\mscr\mathscr

\newcommand{\Z}{\mathbb{Z}}
\newcommand{\Q}{\mathbb{Q}}
\newcommand{\R}{\mathbb{R}}
\newcommand{\C}{\mathbb{C}}

\title{On exact sequences of Hodge theoretic fundamental groups}
\author{Simon Shuofeng Xu}
\date{\today}

\begin{document}

\maketitle
\setcounter{tocdepth}{1}

\begin{abstract}
	The goal of this paper is to first define a Hodge theoretic fundamental group for smooth connected complex algebraic varieties and then prove and study a right exact sequence of Hodge theoretic fundamental groups associated to a smooth projective family of algebraic varieties $f\colon X\to B$. In particular, we study when this right exact sequence is exact, relate this question to some prior results in non-abelian Hodge theory, and give an obstruction to splitting in terms of \'{e}tale fundamental groups. The main examples we consider in this note is the universal curve $f\colon \mc{C}_g\to \mc{M}_g$ and the moduli space of degree $1$ line bundles on the universal curves $p\colon \Pic^1_{\mc{C}_g/\mc{M}_g}\to \mc{M}_g$.
\end{abstract}
\tableofcontents

\section{Introduction}
In this note, we define and study a notion of a Hodge theoretic fundamental group. Let $X$ be a smooth connected complex algebraic variety. Consider the category of $\VMHS(X)$ whose objects are admissible, graded-polarizable $\Q$-variations of mixed Hodge structure on $X$ whose underlying local systems are defined over $\Z$. Any choice of a base point $x\in X(\C)$ defines a fiber functor $F_x$ from this category to the category of $\Q$-vector spaces, and the pair $(\VMHS(X),F_x)$ is a $\Q$-linear neutral Tannakian category. 

\begin{definition}
	We define the \textit{Hodge theoretic fundamental group} of $X$ with respect to some base point $x\in X(\C)$ to be the Tannakian fundamental group $\pi_1(\VMHS(X),F_x)$.
\end{definition}
\begin{remark}
	The term ``Hodge theoretic fundamental group" has appeared in the literature before (see e.g. \cite{Ara10} and \cite{Jac22}) and it was used to refer to a $\pi_1(\LS^{\text{hdg}}(X))$, an object which we are about to introduce, except the local systems considered by Arapura does not necessarily have integral structures. As we will see, our Hodge theoretic fundamental is bigger and contains $\pi_1(\LS^{\text{hdg}}(X))$ as a normal subgroup.\footnote{We thank the anonymous reviewer for alerting us to the paper of Arapura \cite{Ara10}.}
\end{remark}


Let $\LS(X)$ be the category of $\Q$-local systems on $X$ with an underlying integral structure. There's an obvious forgetful functor of Tannakian categories from $\VMHS(X)\to \LS(X)$, and the Tannakian subcategory generated by the essential image of this functor is the subcategory of $\LS(X)$ whose objects are local systems that are subquotients of local systems underlying variations of mixed Hodge structure in $\VMHS(X)$. Following the notation in \cite{DE22}, we denote this subcategory $\LS^\text{Hdg}(X)$, and via Tannakian duality, we have a closed embedding $\pi_1(\LS^\text{Hdg}(X),F_x)\to \pi_1(\VMHS(X),F_x)$. On the other hand, the constant variations of mixed Hodge structure form a sub-Tannakian category of $\VMHS(X)$ which is isomorphic to the category of graded-polarizable $\Q$-mixed Hodge structures $\MHS$, and so Tannakian duality gives us a map $\pi_1(\VMHS(X), F_x)\to \pi_1(\MHS,F)$, where $F$ is the forgetful functor from $\MHS$ to the category of $\Q$-vector spaces. From now on, we will omit the fiber functor when the context is clear. Andr\'{e} \cite[Theorem C. 11]{And21} showed that the following sequence is a (split) short exact sequence $$1\to \pi_1(\LS^\text{Hdg}(X))\to \pi_1(\VMHS(X))\to \pi_1(\MHS)\to 1.$$ See also the work of D'Addezio and Esnault \cite[Corollary 4.7]{DE22} who worked with $\Q$-variations of mixed Hodge structure with no underlying integral structure.\footnote{Arapura also attempted to prove the exactness of this sequence \cite[Theorem 2.6]{Ara10}, but a gap was found in his work by D'Addezio and Esnault \cite[Remark 4.8]{DE22}.} 

\begin{remark}
	We have chosen to work with variation of mixed Hodge structures with an underlying integral structure, mainly because that assumption is needed when we discuss non-abelian Hodge locus in section \ref{section: non-abelian Hodge locus} and it also makes the computation we are going to do in section \ref{section: Pic^1} more tractable.
\end{remark}

In section \ref{section: SES}, we first generalize this short exact sequence to the relative setting. Throughout this paper, by \textit{a smooth projective family of algebraic varieties}, we mean a smooth projective morpshim $f\colon X\to B$ between smooth, connected complex algebraic varieties with connected fibers. Fix some base point $b\in B$ and $x\in X_b:=f^{-1}(b)$. We have the following proposition:
\begin{theorem}
	 The following natural sequence of Hodge theoretic fundamental groups induced by $f\colon X\to B$ is an exact sequence of affine group schemes: \begin{equation}\label{Hodge SES}
		\pi_1(\LS^\text{Hdg}(X_b))\to \pi_1(\VMHS(X))\to \pi_1(\VMHS(B))\to 1
	\end{equation}
\end{theorem}

For the rest of the paper, we further study this sequence $(\ref{Hodge SES})$. In particular, we focus on the questions of when the first map $\pi_1(\LS^\text{Hdg}(X_b))\to \pi_1(\VMHS(X))$ is a closed embedding, and when the sequence has no section. The main examples we will be working with is the universal curve $f\colon \mc{C}_g\to \mc{M}_g$ of genus $g$ and the moduli space of degree $1$ line bundles on the universal curve $p\colon \Pic^1_{\mc{C}_g/\mc{M}_g}\to \mc{M}_g$. Here we view these Deligne-Mumford stacks as orbifolds so we may talk about their (orbifold) fundamental group, and local systems as well as variations of (mixed) Hodge structure on them (for an exposition of the orbifold approach to $\mc{M}_g$, see \cite[Chapter XII, section 4]{ACG11}). In particular, we prove the following theorem
\begin{theorem}\label{Main theorem}\hfill
	\begin{enumerate}
		\item For any $g\geq 2$, there exists some curve $C$ of genus $g$ such that the sequence $$1\to \pi_1(\LS^\text{Hdg}(C))\to \pi_1(\VMHS(\mc{C}_g))\to \pi_1(\VMHS(\mc{M}_g))\to 1$$ is not left exact (see Theorem \ref{Theorem: non-injectivity following Lam});
		\item When $g\geq 3$, the map $\pi_1(\VMHS(\mc{C}_g))\to \pi_1(\VMHS(\mc{M}_g))$ induced by $f\colon \mc{C}_g\to \mc{M}_g$ and the map $\pi_1(\VMHS(\Pic^1_{\mc{C}_g/\mc{M}_g}))\to \pi_1(\VMHS(\mc{M}_g))$ induced by $p\colon \Pic^1_{\mc{C}_g/\mc{M}_g}\to \mc{M}_g$ do not splits (see Corollary \ref{Corollary: Hodge theoretic section question for M_g}).
		\end{enumerate}
\end{theorem}

Now part $(1)$ of Theorem \ref{Main theorem} says that in general the sequence (\ref{Hodge SES}) is not short exact. However, one can also show that at least after passing to the pro-reductive quotient, under some assumptions, this sequence is exact for an analytically very general choice of base point $b\in B$. More precisely, we prove the following
\begin{theorem}
	Suppose $f\colon X\to B$ is a smooth projective family of algebraic varieties such that $\pi_1^{\text{\'{e}t}}(X_b)$ injects into $\pi_1^{\text{\'{e}t}}(X)$. Let $\mathcal{G}^\text{red}$ be the pro-reductive quotient of an affine group scheme $\mathcal{G}$. Define $$NG(X/B)^\text{red}:=\left\{b\in B:(\pi_1(\LS^\text{Hdg}(X_b)))^\text{red}\to (\pi_1(\VMHS(X)))^\text{red}\text{ is not injective}\right\}.$$ Then $NG(X/B)^{\text{red}}$ is a countable union of closed analytic subsets of $B$.
\end{theorem}
Assuming some standard conjectures in non-abelian Hodge theory, this set $NG(X/B)^\text{red}$ is in fact a countable union of closed algebraic subsets of $B$ (see section \ref{section: Injectivity and non-injectivity}). Therefore, it's natural to ask if this is the case without having to pass to pro-reductive quotients: we can define 
$$NG(X/B):=\left\{b\in B:\pi_1(\LS^\text{Hdg}(X_b))\to \pi_1(\VMHS(X))\text{ is not injective}\right\}.$$ 
\begin{question}
	Is the locus $NG(X/B)$ is always a countable union of closed algebraic subsets of $B$?
\end{question}
In fact, we show that this locus can be empty even in non-isotrivial families:
\begin{theorem}
	For any $g\geq 2$ and any curve $C$ of genus $g$, $\pi_1(\LS^\text{Hdg}(\Pic^1(C))$ may be identified with $\mathbb{G}_a^{2g}\times \widehat{\Z}^{2g}$ and the sequence $$1\to \mathbb{G}_a^{2g}\times \widehat{\Z}^{2g}\to \pi_1(\VMHS(\Pic^1_{\mc{C}_g/\mc{M}_g}))\to \pi_1(\VMHS(\mc{M}_g))\to 1$$ is always short exact.
\end{theorem}

Part $(2)$ of Theorem \ref{Main theorem} is interesting as it provides some examples for which a ``Hodge theoretic question" has an affirmative answer (see also \cite{Xu24} for a slightly different formulation of the Hodge theoretic section question): Let $f\colon X\to B$ be a smooth projective family of algebraic varieties. Any algebraic section $s\colon B\to X$ induces a section of $\pi_1(\VMHS(X))\to \pi_1(\VMHS(B))$. Then, roughly speaking, the Hodge theoretic section question asks if group theoretic sections of $\pi_1(\VMHS(X))\to \pi_1(\VMHS(B))$ all come from and are determined by algebraic sections. In particular, if $\pi_1(\VMHS(X))\to \pi_1(\VMHS(B))$ has no section, then the Hodge theoretic section question always has an affirmative answer (in a trivial way). Now part $(2)$ of Theorem \ref{Main theorem} is proven by relating sections of Hodge theoretic fundamental groups to sections of \'{e}tale fundamental group. 
\begin{theorem}
	Let $f\colon X\to B$ be a smooth projective family of algebraic varieties. If the map $\pi_1(\VMHS(X))\to \pi_1(\VMHS(B))$ admits a group theoretic section, then so does the corresponding map of \'{e}tale fundamental groups $\pi_1^{\text{\'{e}t}}(X)\to \pi_1^{\text{\'{e}t}}(B)$.
\end{theorem}
This is a simple consequence of an observation of Andr\'{e}. We will give a different (but morally related) argument. It also allows us to revisit the formulation of the Hodge theoretic section question.

\textbf{Acknowledgement}. I would like to first thank my advisor Daniel Litt for many inspiring discussions and his constant support. This work owes a huge intellectual debt to many beautiful  prior work in Hodge theory. While it's impossible to list all of them, I would like to in particular highlight the work of Simpson on non-abelian Hodge locus \cite{Sim97}, the work of Esnault-Kerz on non-abelian analogue of the theorem of the fixed parts \cite{EK24}, recent works of Landesman, Litt and Lam  on finite orbits of mapping class groups, and the work of Hain-Matsumoto \cite{HM05}. Finally, I would also like to thank Daniel Litt, Charlie Wu, Moritz Kertz, H\'{e}l\`{e}ne Esnault, Marco D'Addezio, Emil Jacobsen and an anonymous reviewer for helpful comments and discussions on a preliminary draft of the paper.

\section{Exact sequence of Hodge theoretic fundamental groups}\label{section: SES}
In this section, our goal is to generalize the work of \cite[Theorem C. 11]{And21} to the relative setting:
\begin{theorem}\label{Thm: exact sequence of Hodge SES}
	Let $f\colon X\to B$ be a smooth projective family of algebraic varieties. The following natural sequence of Tannakian fundamental groups 
	$$\pi_1(\LS^{\text{Hdg}}(X_b))\to \pi_1(\VMHS(X))\to \pi_1(\VMHS(B))\to 1$$ is exact.
\end{theorem}

To prove this theorem, we need to be able to check properties of maps between affine group schemes via the pullback functors between the corresponding category of representations. To this end, we need these facts from the theory of Tannakian categories, which are known to experts and can all be found in \cite{DM82} and \cite[Appendix A]{DE22}. We include them to make this paper more self-contained.

Consider a sequence of affine group schemes over $\Q$: $$\mathcal{F}\to \mathcal{G}\to \mathcal{H}.$$ We would like to understand when the map $\mc{F}\to \mc{G}$ is a closed immersion, when the map $\mc{G}\to \mc{H}$ is faithfully flat, and when the sequence is exact in the middle. To this end, we recall the following proposition:
\begin{proposition}\label{Prop: Tannakian injectivity and surjectivity}
	Let $\phi\colon \mc{F}\to \mc{G}$ be a homomorphism of affine group schemes over $k$, and $\phi^*$ the induced pullback functor $\phi^*\colon \textbf{Rep}_k(\mc{G})\to \textbf{Rep}_k(\mc{F})$. Then 
	\begin{enumerate}
		\item $\phi$ is faithfully flat if and only if $\phi^*$ is observable (i.e. maps semi-simple objects to semi-simple objects) and fully faithful (see \cite[Proposition A. 6]{DE22});
		\item $\phi$ is a closed immersion if and only if every object of $\textbf{Rep}_k(\mc{F})$ is isomorphic to a subquotient of an object of the form $\phi^*(V)$ for some $V\in \textbf{Rep}_k(\mc{G})$ (see \cite[Proposition 2.21 b]{DM82}).
\end{enumerate}
\end{proposition}
The following proposition gives a criterion for exactness in the middle:
\begin{proposition}\label{Prop: Tannakian exactness}
	Let $$\mc{F}\xrightarrow{\phi} \mc{G}\xrightarrow{\psi}\mc{H}\to 1$$ be a sequence of affine group schemes over $\Q$, where $\psi$ is faithfully flat and $\psi\circ \phi$ is trivial. Then the sequence is exact in the middle if 
	\begin{enumerate}
		\item $\phi^*\colon \textbf{Rep}_k(\mc{G})\to \textbf{Rep}_k(\mc{F})$ is observable (i.e. it maps semi-simple objects to semi-simple objects). 
		\item for every $V\in \textbf{Rep}_k(\mc{G})$, there exists $W\in \textbf{Rep}_k(\mc{H})$ such that the maximal trivial subobject of $\phi^*V$ comes from $\psi^*W\subset V$.
	\end{enumerate}
\end{proposition}
\begin{proof}
	 This is essentially a restatement of \cite[Prop. A. 13]{DE22}. Since that proposition as stated in \cite{DE22} additionally assume that $\phi\colon \mc{F}\to \mc{G}$ is a closed immersion, we explain how to remove that assumption. By \cite[Prop. A. 12]{DE22}, we know that $\phi(\mc{F})$ is a closed normal subgroup of $\mc{G}$. Now applying \cite[Prop. A. 13]{DE22} to the sequence $$1\to \phi(\mc{F})\to \mc{G}\to \mc{H}\to 1$$ gives the desired result.
\end{proof}

This following Hodge theoretic lemma should also be well-known but since we could not find a reference, we include the proof for completeness:
\begin{lemma}\label{Lemma: push-pull}
	Let $f\colon X\to B$ be a smooth projective family of algebraic varieties (i.e. it's a smooth projective morphism between smooth connected complex algebraic varieties with connected fibers). Then for any $\mc{V}\in \VMHS(B)$, the natural map $f_*\circ f^*\mc{V}\to \mc{V}$ is an isomorphism. 
\end{lemma}
\begin{proof}
	Let $\mathbb{V}$ be the underlying local system of $\mc{V}$ and suppose $\rk\mathbb{V}=n$. Then locally $\mathbb{V}\otimes \mc{O}_B$ is isomorphic to $\mc{O}_B^{\oplus n}$. Since $f$ has connected fibre, we know that $f_*f^*\mc{O}_B=f_*\mc{O}_X\cong\mc{O}_B$, so the natural map $f_*f^*\mathbb{V}\otimes \mc{O}_B\to \mathbb{V}\otimes \mc{O}_B$ is an isomorphism of flat bundles and by Riemann-Hilbert, the corresponding map is an isomorphism of local systems. 
	
	Then a mixed version of Schmid's rigidity theorem \cite[Proposition 4.20]{SZ85} says that it's enough to show that for any $b\in B$, the map $f_*f^*\mc{V}\to \mc{V}$ induce isomorphism of mixed Hodge structures between $(f_*f^*\mc{V})_b\cong \mc{V}_b$. This is so since $(f_*f^*\mc{V})_b=H^0(X_b,f^*\mc{V})$ and the desired result follows from the theorem of the fixed parts \cite[Proposition 4.19]{SZ85} and the fact that $f^*\mc{V}$ has trivial monodromy over $X_b$.
\end{proof}
\begin{remark}\hfill
\begin{enumerate}
	\item We did not \textit{explicitly} use the fact that $f$ is smooth and projective. Nevertheless, this assumption is needed to ensure that given a $\mc{W}\in \VMHS(X)$, $f_*\mc{W}\in \VMHS(B)$.
	\item The results in \cite{SZ85} are stated in the case where $B$ is of dimension $1$. However, both statements generalize to higher dimensional settings by reducing to the case of curves (see e.g. \cite[4.3.2.0]{Kat72} and \cite[page 88]{HZ87}).
\end{enumerate}
\end{remark}

Now we are ready to prove the Theorem \ref{Thm: exact sequence of Hodge SES}.
\begin{proof}[Proof of Theorem \ref{Thm: exact sequence of Hodge SES}]
	We first show that the map $\pi_1(\VMHS(X))\to \pi_1(\VMHS(B))$ is faithfully flat. By Proposition \ref{Prop: Tannakian injectivity and surjectivity}, we see that it's enough to show that $f^*$ is observable (i.e. sends semi-simple objects to semi-simple objects) and fully faithful. The fact that $f^*$ is observable follows from the fact that the semi-simple objects in $\VMHS_\Q(X)$ are exactly the direct sums of pure $\Q$-variations of Hodge structure, and the pullback of a pure VHS is still pure. Furthermore, by Lemma \ref{Lemma: push-pull}, we know that $\id$ is isomorphic to $f_*\circ f^*$ and so $f^*$ is full. Finally, given any non-zero map between variations of mixed Hodge structure on $B$, the pull-back of this map will still remain non-zero and so $f^*$ is faithful. This shows that $\pi(f)$ is fully faithful. 

Now we would like to apply Proposition \ref{Prop: Tannakian exactness}. First, $\iota^*\circ f^*$ is certainly trivial, as the objects in the image are local systems underlying a variation of mixed Hodge structure on $X_b$ that are pulled back from a point. Furthermore, by Deligne's semisimplicity theorem \cite[Theorem 4.2.6]{Del71}, the local system underlying any pure $\Q$-variation of Hodge structure is also semisimple, and hence $\iota^*$ is observable. Finally, we claim that the maximal trivial subobject $\mc{U}$ of $\iota^*\mc{V}$ is isomorphic to $\iota^*f^*f_*\mc{V}$. Note that $\iota^*f^*f_*\mc{V}$ is certainly a trivial subobject of $\iota^*\mc{V}$ so by maximality, we have a map $\iota^*f^*f_*\mc{V}\to \mc{U}$. It's an isomorphism because the induced maps on stalks are the isomorphism $(\iota^*f^*f_*\mc{V})_b=(f_*\mc{V})_b\cong H^0(X_b,\iota^*\mc{V})=\mc{U}_b$. Therefore, by Proposition \ref{Prop: Tannakian exactness}, the sequence is exact in the middle.
\end{proof}
\begin{corollary}
	Let $f\colon X\to B$ be a smooth projective family of algebraic varieties as before. The following sequence of Tannakian fundamental groups $$\pi_1(\LS^{\text{Hdg}}(X_b))\to \pi_1(\LS^{\text{Hdg}}(X))\to \pi_1(\LS^{\text{Hdg}}(B))\to 1$$ is exact.
\end{corollary}
\begin{proof}
	Consider the following commutative diagram 
\[\begin{tikzcd}
	& 1 & 1 \\
	{\pi_1(\LS^{\text{Hdg}}(X_b))} & {\pi_1(\LS^{\text{Hdg}}(X))} & {\pi_1(\LS^{\text{Hdg}}(B))} & 1 \\
	{\pi_1(\LS^{\text{Hdg}}(X_b))} & {\pi_1(\VMHS(X))} & {\pi_1(\VMHS(B))} & 1 \\
	& {\pi_1(\MHS)} & {\pi_1(\MHS)} \\
	& 1 & 1
	\arrow[from=1-2, to=2-2]
	\arrow[from=1-3, to=2-3]
	\arrow[from=2-1, to=2-2]
	\arrow["{=}", from=2-1, to=3-1]
	\arrow[from=2-2, to=2-3]
	\arrow[from=2-2, to=3-2]
	\arrow[from=2-3, to=2-4]
	\arrow[from=2-3, to=3-3]
	\arrow[from=3-1, to=3-2]
	\arrow[from=3-2, to=3-3]
	\arrow[from=3-2, to=4-2]
	\arrow[from=3-3, to=3-4]
	\arrow[from=3-3, to=4-3]
	\arrow["{=}", from=4-2, to=4-3]
	\arrow[from=4-2, to=5-2]
	\arrow[from=4-3, to=5-3]
\end{tikzcd}\]The desired result follows as the second row, the second and the third column are all exact.\end{proof}

\begin{remark}\label{Remark: pure Hodge SES}
Given an affine group scheme $\mc{G}$ over $\Q$, recall that its maximal pro-reductive quotient may be viewed as the Tannakian fundamental group associated to the sub-Tannakian category of semi-simple objects in $\text{Rep}(\mc{G})$. In the case of $\VMHS(X)$ (resp. of $\LS^{\text{Hdg}}(X)$, this sub-category is the one generated by polarizable pure $\Q$-variations of Hodge structure over $X$ (resp. local systems underlying such variations). If we write $$\pi_1(\LS^{\text{Hdg,ss}}(X)):=\pi_1(\LS^{\text{Hdg}}(X))^{\text{red}},\quad\quad \pi_1(\VHS(X)):=\pi_1(\VMHS(X))^{\text{red}},$$ we see that by Deligne's semi-simplicity theorem, we get a natural sequence of maps $$\pi_1(\LS^{\text{Hdg,ss}}(X_b))\xrightarrow{\iota_b^{s.s}} \pi_1(\VHS(X))\to \pi_1(\VHS(B))\to 1.$$ The same argument as above shows that given a smooth projective family of algebraic varieties $f:X\to B$, this sequence is exact. Note here that the usage of Deligne's semi-simplicity theorem is necessary, as maps between Tannakian fundamental groups do not in general induce maps between their maximal pro-reductive quotient. 
\end{remark}

\section{Failure of injectivity}\label{section: Injectivity and non-injectivity}
The goal of this section is to understand that given a smooth projective family $f\colon X\to B$, when the map $$\iota_b\colon \pi_1(\LS^\text{Hdg}(X_b))\to \pi_1(\VMHS(X))$$ is a closed immersion. Recall that Proposition \ref{Prop: Tannakian injectivity and surjectivity} part (2) says that $\iota_b$ is a closed immersion if every local system in $\LS^\text{Hdg}(X_b)$ is a subquotient of a local system which extends to one that underlies a graded-polarizable, admissible $\Q$-variations of mixed Hodge structure on all of $X$. Phrased slightly differently, this proposition says that the image of $\iota_b$ is the Tannakian fundamental group associated to the category $\LS^{\text{ext}}(X_b)$ of $\Q$-local systems on $X_b$ with an underlying integral structure that are subquotients of local systems on $X_b$ which extend to local systems on all of $X$ underlying some variation of mixed Hodge structure.\footnote{This observation is due to Moritz Kerz in a private email and we would like to thank him for pointing this out.} 

Now suppose $\mathbb{V}$ is an \textit{irreducible} local system in the category $\LS^{\text{ext}}(X_b)$, and let $\mathbb{W}$ be a local system on $X$ such that $\mathbb{W}$ underlies some variation of mixed Hodge structure and that $\mathbb{V}$ is a subquotient of $\mathbb{W}|_{X_b}$. Let $b_1$ be another point of $B$ and pick a path $\gamma\in B$ from $b$ to $b_1$. Let $x$ be a point in the fiber $X_b$ and $x_1$ a point in the fiber of $X_{b_1}$ and again pick some $\gamma'$ a lift of $\gamma$ to $X$ going from $x$ to $x_1$. These choices then fix isomorphisms $\eta_\gamma\colon \pi_1(B,b_1)\to \pi_1(B,b)$ and $\eta_{\gamma'}\colon \pi_1(X,x_1)\to \pi_1(X,x)$, and they fit into the following commutative diagram:
\[\begin{tikzcd}
	1 & {\pi_1(X_{b_1},x_1)} & {\pi_1(X,x_1)} & {\pi_1(B,b_1)} & 1 \\
	1 & {\pi_1(X_{b},x)} & {\pi_1(X,x)} & {\pi_1(B,b)} & 1
	\arrow[from=1-1, to=1-2]
	\arrow[from=1-2, to=1-3]
	\arrow["\eta_{\gamma'}", from=1-2, to=2-2]
	\arrow[from=1-3, to=1-4]
	\arrow["{\eta_{\gamma'}}", from=1-3, to=2-3]
	\arrow[from=1-4, to=1-5]
	\arrow["{\eta_\gamma}", from=1-4, to=2-4]
	\arrow[from=2-1, to=2-2]
	\arrow[from=2-2, to=2-3]
	\arrow[from=2-3, to=2-4]
	\arrow[from=2-4, to=2-5]
\end{tikzcd}\] where we will abuse the notation and denote $\eta_{\gamma'}$ the restriction of $\eta_{\gamma'}$ to $\pi_1(X_{b_1},x_1)$ as well. For any local system $\mathbb{V}$ on $X_b$ defined via a representation $\rho\colon \pi_1(X_b)\to \GL(V)$, let $\eta_{\gamma'}^*\mathbb{V}$ be the local system on $X_{b_1}$ defined by the representation $$\rho\circ \eta_{\gamma'}\colon \pi_1(X_{b_1},x_1)\to \pi_1(X_b,x)\to \GL(V).$$ 
\begin{lemma}\label{Lemma: extending local systems}
	For any choices of $b_1,\gamma,\gamma'$, we have \hfill
	\begin{enumerate}
	\item the local system $\eta_{\gamma'}^*\mathbb{V}$ underlies a polarizable complex variation of Hodge structure;
	\item the local system $\eta_{\gamma'}^*\mathbb{V}\oplus \eta_{\gamma'}^*\mathbb{V}$ underlies a polarizable $\Q$-variation of Hodge structure with an underlying integral structure, and hence $\eta_{\gamma'}^*\mathbb{V}\in \LS^\text{Hdg}(X_{b_1})$.
	\end{enumerate}
\end{lemma}
\begin{proof}
	Note that $\mathbb{W}|_{X_{b_1}}=\eta_{\gamma'}^*\mathbb{W}|_{X_b}$ so since $\mathbb{V}$ is a subquotient of $\mathbb{W}|_{X_b}$, $\eta_{\gamma'}^*\mathbb{V}$ is a subquotient of $\mathbb{W}|_{X_{b_1}}$. Since $\mathbb{W}$ underlies a variation of mixed Hodge structure, so does $\mathbb{W}|_{X_{b_1}}$. Since $\eta_{\gamma'}^*\mathbb{V}$ is irreducible, it's a sub-local system of $\Gr^W_\bullet (\mathbb{W}|_{X_{b_1}})$ which underlies a pure polarizable $\Q$ and hence $\C$-variation of Hodge structure. Call this variation $\mc{W}$. By \cite[Prop. 1.13]{Del87}, $\eta^*\mathbb{V}$ underlies a polarizable $\C$-VHS $\mc{V}$, which is a sub-variation of $\mc{W}$ viewed as a polarizable $\C$-VHS. This proves the first statement.
	
	To prove the second statement, we consider the $\C$-VHS structure on the flat bundle $(\eta_{\gamma'}^*\mathbb{V}\otimes \mc{O})\oplus (\eta_{\gamma'}^*\mathbb{V}\otimes \mc{O})$, where the Hodge filtration on the first copy of $\eta_{\gamma'}^*\mathbb{V}\otimes \mc{O}$ agrees with the Hodge filtration on $\mc{V}$ and the Hodge filtration on the second copy of $\eta_{\gamma'}^*\mathbb{V}\otimes \mc{O}$ is the conjugate filtration of the one on $\mc{V}$. By construction, we see that the Hodge filtration on $(\eta_{\gamma'}^*\mathbb{V}\otimes \mc{O})\oplus (\eta_{\gamma'}^*\mathbb{V}\otimes \mc{O})$ satisfies $H^{p,q}=\overline{H^{q,p}}$ on each fiber and hence $\eta_{\gamma'}^*\mathbb{V}\oplus \eta_{\gamma'}^*\mathbb{V}$ underlies a polarizable $\Q$-variation of Hodge structure.
\end{proof}
We will use this Lemma to show that $\iota_b$ does not have to be injective in general. 
\begin{theorem}\label{Theorem: non-injectivity following Lam}
	Fix some integer $g\geq 2$. Then there exists a curve $C$ of genus $g$ such that the map $$\iota_{[C]}\colon \pi_1(\LS^{\text{Hdg}}(C))\to \pi_1(\VMHS(\mc{C}_g))$$ is not injective.
\end{theorem}
To prove this, we will find a curve $C$ of genus $g$ equipped with a $\SL_2$-local system $\mathbb{V}\in \LS^{\text{Hdg}}(C)$ satisfying some extra conditions which ensure that there are only finitely many curves of a given genus carrying such a local system. First recall that by the uniformization theorem, any curve $C$ of genus at least $2$ is uniformized by the upper-half plane $\mc{H}$ and choosing a uniformization gives a representation $$\rho\colon \pi_1(C,x)\to \PSL_2(\R),$$ where $\PSL_2(\R)$ is viewed as the group of orientation preserving isometries of $\mc{H}$. Such a representation will be called a uniformization representation. 
\begin{definition}[Arithmetic uniformizing local system]
Let $C$ be a smooth projective curve of genus at least $2$ and $\mathbb{V}$ an $\SL_2$-local system on $C$. We say that $\mathbb{V}$ is an \textit{arithmetic uniformizing} local system if it satisfies the following conditions:
\begin{enumerate}
	\item $\mathbb{V}$ is uniformizing in the sense that the associated monodromy representation $$\rho_{\PGL_2}\colon\pi_1(C,x)\to \SL_2(\R)\to \PSL_2(\R)$$ is conjugate to a uniformization representation $\pi_1(C,x)\to \PSL_2(\R)$ obtained by choosing uniformization of $C$.
	\item the image of $\rho_{\SL_2}\colon \pi_1(C,x)\to \SL_2(\R)$ is an arithmetic Fuchsian group (for a precise definition of arithmetic Fuchsian group, see \cite[Definition 1]{Tak83}).
\end{enumerate}
\end{definition}
We will show that for any integer $g\geq 2$, there exists a curve $C$ carrying an irreducible arithmetic uniformizing local system $\mathbb{V}$ such that $\mathbb{V}$ is not in $\LS^{\text{ext}}(C)$ and so the map $\iota_{[C]}\colon \pi_1(\LS^{\text{Hdg}}(C))\to \pi_1(\VMHS(\mc{C}_g))$ is not injective. To that end, we need the following finiteness lemma which is essentially due to Takeuchi \cite{Tak83}.
\begin{lemma}\label{Lemma: finiteness of arithmetic uniformizing theorem}
		Fix some $g\geq 2$. There are only finitely many of isomorphism classes of curves of genus $g$ carrying an arithmetic uniformizing local system.
\end{lemma}
\begin{proof}
	Since $\mathbb{V}$ is uniformizing, we see that the image determines the isomorphism class of the curve as $C\cong \mc{H}/\rho_{\PGL_2}(\pi_1(C,x))$. Therefore, it's enough to show that the set of conjugacy classes of possible images of arithmetic uniformizing local systems is finite. This follows from \cite[Theorem 2.1]{Tak83} (note that since $C$ is a smooth compact Riemann surface, the image of $\rho_{\PGL_2}$ has no elliptic or parabolic element).
\end{proof}
\begin{proposition}[Lam]\label{Proposition: Lam's theorem rephrased}
	Suppose $C$ is a curve of genus $g=2$ carrying an irreducible arithmetic uniformizing local system $\mathbb{V}$. Then $\mathbb{V}$ is not contained in $\LS^{\text{ext}}(C)$.
\end{proposition}
\begin{proof}
Assume for the sake of contradiction that $\mathbb{V}\in \LS^{\text{ext}}(C)$. Since it's irreducible, we know by Lemma \ref{Lemma: extending local systems} that for any other curve $C'$ and any path $\gamma$ connecting $C$ and $C'$ in $\mc{C}_g$, $\eta_{\gamma}^*\mathbb{V}$ underlies a polarizable complex variation of Hodge structure on $C'$. Let $(\mc{E},\nabla)$ be the flat bundle associated to $\eta_{\gamma}^*\mathbb{V}$. It has a Hodge filtration $F^\bullet$, and let $\mc{L}:=F^1\mc{E}$. Consider the short exact sequence $$0\to \mc{L}\to \mc{E}\to \mc{E}/\mc{L}\to 0.$$ Taking the determinant, we see that $\mc{E}/\mc{L}$ has to be $\mc{L}^{-1}$. As $\eta_{\gamma}^*\mathbb{V}$ underlies a polarizable $\C$-variation of Hodge structure, by \cite[Proposition 2.1]{Lam22}, we know that $\deg\mc{E}/\mc{L}=\deg\mc{L}^{-1}<0$ and so $\deg\mc{L}>0$.

Consider the Higgs map associated to this flat bundle $$\mc{L}\to \mc{E}/\mc{L}\otimes \Omega_{C}^1=\mc{L}^{-1}\otimes \Omega_C^1.$$ This map is non-zero as the local system is irreducible, and hence we have the inequality $$0< \deg\mc{L}\leq \deg\mc{L}^{-1}\otimes \Omega_C^1=2-\deg\mc{L}$$ from which we see that $\deg \mc{L}=\deg \mc{L}^{-1}\otimes \Omega_C^1=1$. It follows that the Higgs map is an isomorphism and so $\mc{L}$ is a square root of $\Omega_C^1$. In particular, the associated Higgs bundle, and hence the local system is uniformizing by the work of Hitchin \cite[Example 1.5]{Hit87}. 

On the other hand, the image of the monodromy representation $\rho_{\SL_2}$ associated to $\eta^*_{\gamma}$ agrees with the one associated to $\mathbb{V}$ and so is still an arithmetic Fuchsian group. In particular, $\eta_{\gamma}^*\mathbb{V}$ is an arithmetic uniformizaing local system. As $\mc{M}_g$ is connected, we see that every curve then must carry an arithmetic uniformizing local system, contradicting Lemma \ref{Lemma: finiteness of arithmetic uniformizing theorem}.
\end{proof}
\begin{corollary}
	There exists a curve $C$ of genus $2$ such that the map $$\iota_{[C]}\colon \pi_1(\LS^{\text{Hdg}}(C))\to \pi_1(\VMHS(\mc{C}_2))$$ is not injective.
\end{corollary}
\begin{proof}
	It's enough to find a genus $2$ curve carrying an irreducible arithmetic uniformizing local system. It's known that the canonical $\SL_2$-local system attached to a genus $2$ Shimura curve satisfies these conditions.
\end{proof}
\begin{remark}
	We've attributed Proposition \ref{Proposition: Lam's theorem rephrased} to Lam because although the statement is different, the proof is essentially the same as the proof of \cite[Theorem 1.1]{Lam22}.
\end{remark}

We would like to adapt the same argument and make it work in the case where $g>2$. We set up some notations. Let $\mc{H}_d$ be the moduli space of degree $d$ covers of a genus $2$ curve. It admits a map onto $\mc{M}_2$. Let $S$ and $S'$ be two curves of genus $2$ and pick a path $\gamma$ from $[S]$ to $[S']$ inside $\mc{M}_2$. We've seen that by choosing such a path and a lift of this map to $\mc{C}_2$, we may fix an isomorphism $$\eta_\gamma\colon \pi_1(S',x')\to \pi_1(S,x).$$ Now for any subgroup $H$ of $\pi_1(S,x)$ of index $d$, we see that $H':=\eta_\gamma^{-1}(H)$ is a finite index subgroup of $\pi_1(S',x')$ of index $d$ as well so it corresponds a unique finite \'{e}tale cover $C'$ of $S'$. Now suppose we have a local system on $\pi_1(S,x)$ with associated monodromy representation $\rho\colon \pi_1(S,x)\to \GL(V)$, we define $$\begin{aligned}
	\eta^*_\gamma \mathbb{V}&:=\text{ the local system on }S' \text{ associated to }\pi_1(S',x')\to \pi_1(S,x)\to \GL(V)\\ 
	f^*\mathbb{V} &:=\text{ the local system on }S' \text{ associated to }H\to \pi_1(S,x)\to \GL(V)\\
	\eta^*_\gamma f^*\mathbb{V}&:=\text{ the local system on }S' \text{ associated to }H'\to H\to \pi_1(S,x)\to \GL(V)
\end{aligned}$$
By construction, it's clear that $f^*$ commutes with $\eta_{\gamma}^*$ so $$\eta_{\gamma}^*f^*\mathbb{V}=f^*\eta_{\gamma}^*\mathbb{V}.$$ Furthermore, set $$\begin{aligned}
	\mc{E}_{S}&:=\text{the flat bundle on }S\text{ associated to }\mathbb{V}\\
	\mc{E}_{S'}&:=\text{the flat bundle on }S'\text{ associated to }\eta^*_\gamma\mathbb{V}\\
	\mc{E}_{C}&:=\text{the flat bundle on }  C\text{ associated to }f^*\mathbb{V}\\
	\mc{E}_{C'}&:=\text{the flat bundle on }C'\text{ associated to }\eta^*_\gamma f^*\mathbb{V}\\
\end{aligned}$$

\begin{proof}[Proof of Theorem \ref{Theorem: non-injectivity following Lam}]
	Let $S$ be a genus $2$ Shimura curve and $\mathbb{V}$ the canonical $\SL_2$-local system on $S$. Since it's of genus $2$, we know that for any $g\geq 3$, there exists a genus $g$ finite \'{e}tale cover $C$ of $S$. Fix such a cover and the $H$ will be the finite index subgroup of $\pi_1(S,x)$ corresponding to this cover $C$. Assume for the sake of contradiction that $f^*\mathbb{V}\in \LS^{\text{ext}}(C)$. We will show that $\eta_{\gamma}^*f^*\mathbb{V}$ is always an arithmetic uniformizing local system for any choice of $S'$ and path $\gamma$. This then contradicts Lemma \ref{Lemma: finiteness of arithmetic uniformizing theorem}, as the forgetful map from $\mc{H}_d\to \mc{M}_{d+1}$ is of infinite image.
	
	The image of the associated monodromy representation is certainly still an arithmetic Fuchsian group, as it's of finite index inside the image of the monodromy representation associated to $\mathbb{V}$. Therefore, it's enough to show that it's uniformizing. We claim that the same argument as the one in the proof of Proposition \ref{Proposition: Lam's theorem rephrased} will work: notice that the only place where we need the genus $2$ assumption is that we need it to do the degree calculation to show that the Higgs map is an isomorphism. However, in our setup, we can do the degree calculation using the fact that the Hodge filtration on $\mc{E}_{C'}$ is pulled back from the Hodge filtration on $\mc{E}_{S'}$. In particular, $$\deg F^1\mc{E}_{C'}=d \cdot \deg F^1\mc{E}_{S'}.$$ Furthermore, since the map is \'{e}tale, we know that $f^*\Omega_{S'}=\Omega_{C'}$ and so similarly, $$\deg \left((\mc{E}_{C'}/F^1\mc{E}_{C'})\otimes \Omega_{C'}^1\right)=d\cdot \deg \left((\mc{E}_{S'}/F^1\mc{E}_{S'})\otimes \Omega_{S'}^1\right).$$ It follows that $\deg F^1\mc{E}_{C'}=\deg \left((\mc{E}_{C'}/F^1\mc{E}_{C'})\otimes \Omega_{C'}^1\right)$ and the Higgs map  is again an isomorphism as desired.
\end{proof}
\begin{remark}
	One may also use \cite[Theorem 1.2.5]{LL24} to prove Theorem \ref{Theorem: non-injectivity following Lam} (at least when $g$ is large). We've decided to adapt Lam's argument since it gives us a way of concretely writing down a local system not contained in $\LS^{\text{ext}}(C)$.
\end{remark}

\section{Non-abelian analogue of the Hodge exceptional locus}\label{section: non-abelian Hodge locus}
In this section, we study the set where the map $\iota_b\colon \pi_1(\LS^\text{Hdg}(X_b))\to \pi_1(\VMHS(X))$ is not injective. We introduce the following definition:

\begin{definition}\label{Definition: NG locus}
	Let $f\colon X\to B$ be given as above. Define $$\begin{aligned}NG(X/B)&:=\{b\in B: \iota_b\colon \pi_1(\LS^\text{Hdg}(X_b))\to \pi_1(\VMHS(X))\text{ is not injective}\}\end{aligned}$$
\end{definition}

In view of Proposition \ref{Prop: Tannakian injectivity and surjectivity} part (2), we see that the points $b\in NG(X/B)$ can be thought of as the set of points over which the fibers $X_b$ have ``extra" local systems that underlie $\Q$-variations of Hodge structure. This can be viewed as a non-abelian analogue of the Hodge exceptional locus studied in Cattani-Deligne-Kaplan \cite{CDK95}, which is the locus of points over which the fibers $X_b$ have ``extra" Hodge classes. Because Hodge filtrations vary holomorphically in a polarizable $\Q$-variation of Hodge structure, it's not hard to see that the Hodge exceptional locus is a countable union of closed analytic subsets (see e.g. \cite[section 3.1]{Voi13} for an detailed exposition). Inspired by these results, we conjecture the following:
\begin{conjecture}
	Let $f\colon X\to B$ be a smooth projective family of algebraic varieties. Then $NG(X/B)$ is a countable union of analytically closed subsets of $B$. 
\end{conjecture}
We provide some evidence for this conjecture by showing that under some assumption, if we work with pure variations of Hodge structure, then the conjecture is true. More precisely, consider the set
$$NG(X/B)^{\text{red}}:=\{b\in B: \iota_b^{ss}\colon \pi_1(\LS^\text{Hdg,ss}(X_b))\to \pi_1(\VHS(X))\text{ is not injective}\}.$$ We show that

\begin{theorem}\label{Theorem: countable union of closed analytic subsets}
	Suppose $f\colon X\to B$ is a smooth projective family of algebraic varieties such that $\pi_1^{\text{\'{e}t}}(X_b)$ injects into $\pi_1^{\text{\'{e}t}}(X)$. Then $NG(X/B)^\text{red}$ is a countable union of closed analytic subsets of $B$. 
\end{theorem}
We prove this theorem by relating this set $NG(X/B)^{\text{red}}$ to Simpson's non-abelian Hodge locus \cite[section 12]{Sim97} (also known as the non-abelian Noether-Lefschetz locus). We briefly recall the definition of non-abelian Hodge locus and some known results about it. Let $f\colon X\to B$ be a smooth projective family of algebraic varieties. Let $\mc{M}_{\text{Dol}}(X/B)$ be the relative coarse moduli space for semistable principal Higgs bundles $(\mc{E},\phi)$ with vanishing rational Chern classes. Note that two semistable Higgs bundles with the same semi-simplification will correspond to the same point in this moduli space. Let $\mc{M}_{b}(X/B)$ the relative coarse moduli space of vector bundles with flat connections. From non-abelian Hodge theory (see e.g. \cite{Sim92}), we get two things: 
\begin{enumerate}
	\item an homeomorphism between $\mc{M}_{\text{Dol}}(X/B)$ and $\mc{M}_{\text{dR}}(X/B)$ known as the non-abelian Hodge correspondence \cite[page 12]{Sim97};
	\item a $\mathbb{G}_m$-action on $\mc{M}_{\text{Dol}}(X/B)$ whose fixed points are exactly those Higgs bundles that correspond to systems of Hodge bundles associated to a $\C$-variation of Hodge structure \cite[Corollary 4.2]{Sim92}.
\end{enumerate}
Let $V$ be the set of fixed points of the action of $\mathbb{G}_m$ on $\mc{M}_{\text{Dol}}(X/B)$ and let $V_{\text{dR}}$ be its image in $\mc{M}_{\text{dR}}(X/B)$ under the non-abelian Hodge correspondence. Now let $\mc{M}_{\text{dR}}(X/B,\Z)$ be the subset of $\mc{M}_{\text{dR}}(X/B)$ which over each fiber $X_b$ correspond to a flat bundle whose associated monodromy representation has an integral structure. Simpson's non-abelian Hodge locus $NAHL(X/B)$ is then defined as the intersection of these two sets $$NAHL(X/B):=V_{\text{dR}}\cap \mc{M}_{\text{dR}}(X/B,\Z).$$ 
\begin{remark}\label{Remark: remark on connected components and isomonodromic deformation}
	Note that because $\GL_n(\Z)$ is discrete, $\mc{M}_{\text{dR}}(X/B,\Z)$, and hence $NAHL(X/B)$, are horizontal with respect to the isomonodromic foliation. In particular, two local systems are in the same connected components of $NAHL(X/B)$ if and only if they are related via an isomonodromic deformation (for the definition of an isomonodromic deformation, see \cite[Definition 3.4.3]{LL24}).
\end{remark}
We need the following important fact concerning Simpson's non-abelian Hodge locus:
\begin{proposition}[Theorem 12.1 in \cite{Sim97}]\hfill
\begin{enumerate}
	\item Simpson's non-abelian Hodge locus $NAHL(X/B)$ has a unique structure of a reduced analytic variety such that the inclusion $NAHL(X/B)\hookrightarrow \mc{M}_{\text{dR}}(X/B)$ is complex analytic.
	\item The canonical map $\mc{M}_{\text{dR}}(X/B)\to B$ restricts to a proper map $NAHL(X/B)\to B$. 
\end{enumerate}
\end{proposition}
We need to relate Simpson's non-abelian Hodge locus and $NG(X/B)^{\text{red}}$.
\begin{lemma}\label{Lemma: description of NG}
	If $f\colon X\to B$ is a smooth projective family of algebraic varieties such that $\pi^{\text{\'{e}t}}_1(X_b)$ injects into $\pi_1^{\text{\'{e}t}}(X)$, then the set $NG(X/B)^{\text{red}}$ is the image of the components of $NAHL(X/B)\subset \mc{M}_{\text{dR}}(X/B)$ which do not surject onto $B$ under the canonical map $\mc{M}_{\text{dR}}(X/B)\to B$.
\end{lemma}
\begin{proof}
	Let $\mathbb{V}$ be a local system on $X_b$ whose isomorphism class is contained in one of the components of $NAHL(X/B)$ that do not surject onto $B$. Since $\mathbb{V}$ is semi-simple, without loss of generality, we may assume that $\mathbb{V}$ is irreducible. Then pick any $b_1$ not in the image of this component in $B$ and any path $\gamma$ from $b$ to $b_1$. Using the notation as before, let $\eta_\gamma^*\mathbb{V}$ be the corresponding local system on $X_{b_1}$. By definition, $\eta_\gamma^*\mathbb{V}$ is not contained in $NAHL(X/B)$ and so cannot underlie a polarizable $\Q$-variation of Hodge structure with an underlying integral structure. Then by the converse of Lemma \ref{Lemma: extending local systems} part 2, $\mathbb{V}$ is not contained in $\LS^{\text{ext}}(X_b)$. Hence, $b\in NG(X/B)^{\text{red}}$, and images of the components which do not surject onto $B$ are contained in $NG(X/B)^{\text{red}}$.
	
	It remains to show that if $\mathbb{V}$ is a local system on $X_b$ which is contained in a connected component that does surject onto all of $B$, then $\mathbb{V}$ is a subquotient (equivalently sub-object) of some local system which does extend to a variation of Hodge structure on all of $X$. This is essentially done in \cite{EK24}: by the proof of \cite[Theorem 1.1]{EK24} in section 6.3, (in particular, the direction of 5)$\implies$ 1)), we know that such a component must be finite, \'{e}tale over the base and hence the local system $\mathbb{V}$ is of finite orbit under the action of $\pi_1(B)$. Then \cite[Theorem 1.1]{EK24} (in particular the equivalence of 1) and 3)) says that $\mathbb{V}$ extends to a local system $\mathbb{V}'$ on 	some finite cover $X'\to X$ that underlies a $\Q$-variation of Hodge structure on $X'$ with an underlying integral structure. Then $f_*\mathbb{V}'$ is a local system underlying some $\Q$-variation of Hodge structure on $X$ whose restriction contains $\mathbb{V}$ as a sub-local system as desired.
\end{proof}
\begin{proof}[Proof of Theorem \ref{Theorem: countable union of closed analytic subsets}]
	The desired result follows immediately from Lemma \ref{Lemma: description of NG} since the map $f\colon NAHL(X/B)\to B$ is proper and hence closed.
\end{proof}

\begin{remark}\hfill
\begin{enumerate}
	\item The assumption that $\pi^{\text{\'{e}t}}_1(X_b)$ injects into $\pi_1^{\text{\'{e}t}}(X)$ is needed to apply the results of Esnault-Kerz. It's reasonable to conjecture that this description is still true without this assumption. 
	\item This assumption however is important from a purely group theoretic perspective: The main group theoretic result \cite[Theorem 4.1]{EK24} of Esnault-Kerz says that if $H$ is a subgroup of $G$ such that $\widehat{H}$ injects into $\widehat{G}$, then a semi-simple representation of $H$ extends to a semi-simple representation of some finite index subgroup $G'$ of $G$ containing $H$ if and only if this representation is of finite orbit under the action of $G/H$. This result is not true if we don't make any assumption. Indeed, this statement with no assumptions imposed would imply that any extension of residually finite group by residually finite group is residually finite, as we explain now: consider the extension of groups $$1\to H\to G\to G/H\to 1,$$ where both $H$ and $G/H$ are residually finite. To show that $G$ is residually finite, it's enough to show that every non-zero element $g\in H\subset G$ is non-zero in some finite quotient of $G$. Since $H$ is residually finite, we know that $H$ has some finite quotient in which $g$ remains non-zero. As every finite group is linear, we see that there's a representation $\rho$ of $H$ of finite image such that $\rho(g)\neq 0$. Since every representation of finite image is automatically of finite orbit under the action of $G/H$, then under the assumption, we may extend it to a representation $\rho'$ of some finite index subgroup $G'$ of $G$ containing $H$. In particular $\rho'(g)\neq 0$. Let $K'$ be the kernel of $\rho'$. It's also of finite index in $G$ and hence the intersection of all subgroups conjugate to $K'$ is a normal, finite index subgroup. Furthermore, $g$ is again non-zero in the corresponding finite quotient, and so $G$ is residually finite as claimed. \footnote{The original argument had a mistake and we thank Professor Esnault for pointing out the mistake.} On the other hand, Millson \cite{Mil79} has constructed an example of an extension of residually finite group by a finite group that's not residually finite. Hence, some assumption is needed for the main group theoretic result of Esnault-Kerz to hold.
\end{enumerate}
\end{remark}

Let's return to the set $NG(X/B)$. We know that in the abelian setting, the Hodge conjecture famously implies that the Hodge exceptional locus is in fact algebraic and this is eventually proven unconditionally in \cite{CDK95}. Similarly in the non-abelian setting, the non-abelian version of the Hodge conjecture \cite[Conjecture 12.4]{Sim97} would also imply that Simpson's non-abelian Hodge locus, and therefore $NG(X/B)^{\text{red}}$, is algebraic. It seems reasonable to ask if the same is true with no assumptions and without having to pass to pro-reductive quotient:
\begin{question}
	Let $f\colon X\to B$ be a smooth projective family of algebraic varieties. Is $NG(X/B)$ always a countable union of algebraic subsets of $B$?
\end{question}

\section{Case study: moduli space of degree $1$ line bundles on the universal curve}\label{section: Pic^1}
In this section, we study the example of moduli space of degree $1$ line bundles on the universal curve $p\colon \Pic^1_{\mc{C}_g/\mc{M}_g}\to \mc{M}_g$. Let $C$ be a smooth projective curve of genus $g$. We first give an explicit description of $\pi_1(\VMHS(\Pic^1(C)))$. In fact, this description works for any algebraic variety $A$ whose topological fundamental group is a finitely generated free abelian group, so we first work in that generality.

\begin{proposition}\label{Prop: Tannakian fundamental group for A.V.}
	Let $A$ be smooth connected complex algebraic variety with $\pi_1(A)=\Z^{k}$. For any choice of base point, we have $\pi_1(\LS^{\text{Hdg}}(A))=\mathbb{G}_a^{k}\times \widehat{\Z}^{k}$.
\end{proposition}
The following lemma is probably well-known to experts, and we only prove it for completeness.
\begin{lemma}\label{Lemma: monodromy of semisimple local systems on A.V.}
	Let $A$ be a smooth connected complex algebraic variety with abelian fundamental group and $\mathbb{V}$ a $\Q$-local system underlying a polarizable $\Q$-variation of Hodge structure with an underlying integral structure. Then $\mathbb{V}$ is of finite monodromy. 
\end{lemma}
\begin{proof}
Let $\mathbb{V}$ be such a local system. Then $\mathbb{V}\otimes \C$ underlies a polarizable $\C$-variation of Hodge structure and since $\pi_1(A)$ is abelian, it splits into irreducible rank $1$ summands $\mathbb{V}\otimes \C=\bigoplus \mathbb{V}_i$. By \cite[Prop. 1.13]{Del87}, each $\mathbb{V}_i$ underlies a polarized $\C$-variation of Hodge structure. In particular, the polarization has to be positive-definite and hence $\mathbb{V}_i$ all have unitary monodromy. It follows that $\mathbb{V}$ has unitary monodromy as well. Finally, since it in addition has an associated integral structure, it's finite as desired.
\end{proof}
\begin{proof}[Proof of Proposition \ref{Prop: Tannakian fundamental group for A.V.}]
	We need to show that the category of representations of $\mathbb{G}_a^{k}\times \widehat{\Z}^{k}$ is equivalent to $\LS^{\text{Hdg}}(A)$ and the equivalence is compatible with tensor products and fiber functors. Let $\mathbb{V}$ be a local system in $\LS^{\text{Hdg}}(A)$. Then it's given by $k$ commuting linear operators $L_i$ on a fiber $\mathbb{V}_x$. We know it is a subquotient of some local system which underlies a $\Q$-variation of mixed Hodge structure with an underlying integral structure, so the semi-simplification $\Gr_\bullet\mathbb{V}$ is a summand of some $\Q$-local system underlying a polarizable $\Q$-VHS with an underlying integral structure. Then Lemma \ref{Lemma: monodromy of semisimple local systems on A.V.} tells us that $\Gr_\bullet\mathbb{V}$ is of $\mathbb{V}$ finite monodromy, and so by the Jordan-Chevalley decomposition, we may write $L_i=U_iN_i$, where $U_i$ is of finite order and $N_i$ is unipotent. Since $U_i$ commutes with $N_i$, we see that we get a representation of $\mathbb{G}_a^{k}\times \widehat{\Z}^{k}$. This defines a functor $$\Phi\colon  \LS^{\text{Hdg}}(A)\to \text{Rep}(\mathbb{G}_a^{k}\times \widehat{\Z}^{k}).$$
	
Conversely, given a representation of $\mathbb{G}_a^{k}\times \widehat{\Z}^{k}$, we need to produce a local system in $\LS^{\text{Hdg}}(A)$. Such a representation is the same as $k$ linear operators $U_i$ of finite order and $k$ unipotent operators $N_i$ that all commute with each other. Hence we get a local system $\mathbb{V}$ of $A$ by assigning each generators of $\Z^{k}$ to $U_iN_i$. We need to show that $\mathbb{V}$ is a subquotient of a local system which underlies a variation of mixed Hodge structure. Since the semi-simplification of $\Gr^\bullet\mathbb{V}$ has finite monodromy, we know that there exists a finite \'{e}tale cover $\chi\colon A'\to A$ such that $\chi^*\mathbb{V}$ is a unipotent local system on $A'$. It follows that $\mathbb{V}$ is a sub-local system of $\chi_*\chi^*\mathbb{V}$. Therefore, it's enough to show that every unipotent local system $\mathbb{V}$ is a subquotient of some local system which underlies a variation of mixed Hodge structure. This follows from the work of D'Addezio-Esnault \cite[Theorem 4.4]{DE22} as well as the work of Jacobsen \cite[Theorem 7.2]{Jac22}, who showed that the full subcategory of local systems that are subquotients of local systems underlying graded-polarizable, admissible $\Q$ variation of mixed Hodge structure is closed under taking extensions.
	
Finally, it's clear that two constructions are inverses of each other, and are compatible with taking tensor products and fiber functors, so the desired result follows from Tannakian duality.
\end{proof}
\begin{remark}
	One may also prove that every unipotent local system is a subquotient of some local system underlying a variation of mixed Hodge structure by using the concrete description of unipotent variations of mixed Hodge structure by Hain-Zucker \cite[Theorem 1.6]{HZ87}.
\end{remark}
Now we return to the case of $p\colon \Pic^1_{\mc{C}_g/\mc{M}_g}\to \mc{M}_g$. The main theorem of this section is the following:
\begin{theorem}\label{Theorem: injectivity for Pic^1}
	Fix some $g\geq 2$, and let $C$ be a smooth projective curve of genus $g$. Then the following sequence is exact $$1\to \pi_1(\LS^\text{Hdg}(\Pic^1(C)))\to \pi_1(\VMHS(\Pic^1_{\mc{C}_g/\mc{M}_g})\to \pi_1(\VMHS(\mc{M}_g))\to 1$$
\end{theorem}
Note that to show that the map $\pi_1(\LS^{\text{Hdg}}(\Pic^1(C)))\to \pi_1(\VMHS(\Pic^1_{\mc{C}_g/\mc{M}_g}))$ is injective, it's enough to find representations of $\pi_1(\VMHS(\Pic^1_{\mc{C}_g/\mc{M}_g}))$ whose restrictions to $\mathbb{G}_a^{2g}\times \widehat{\Z}^{2g}$ are jointly faithful.

By unpacking the identification of $\pi_1(\LS^\text{Hdg}(\Pic^1(C)))=\mathbb{G}_a^{2g}\times \widehat{\Z}^{2g}$ in the proof of Proposition \ref{Prop: Tannakian fundamental group for A.V.} and using Tannakian duality, we see that it's enough to find local systems $(\mc{E}_i)_{i\in \mathbb{N}}$ on $\Pic^1_{\mc{C}_g/\mc{M}_g}$ satisfying the following conditions:
\begin{enumerate}
	\item All of the $\mc{E}_i$'s underlie some graded-polarizable, admissible $\Q$-variation of mixed Hodge structure $\Pic^1_{\mc{C}_g/\mc{M}_g}$;
	\item the restrictions of $\mc{E}_0$ to the fiber $\Pic^1(C)$ is unipotent and faithful;
	\item the restrictions of $(\mc{E}_i)_{i\geq 1}$ to the fiber $\Pic^1(C)$ are of finite monodromy and are jointly faithful.
\end{enumerate} We first explain how to construct the local system $\mc{E}_0$. This is inspired by the local systems studied by Hain-Matsumoto in \cite{HM05} and the key tool is to use homologically trivial relative cycles. 

Consider $\mc{X}:=\mc{C}_{g}\times_{\mc{M}_g}\Pic^1_{\mc{C}_g/\mc{M}_g}$ together with two projection maps $p_1\colon  \mc{X}\to \mc{C}_g$ and $p_2\colon \mc{X}\to \Pic^1_{\mc{C}_g/\mc{M}_g}$. By definition of $\Pic^1_{\mc{C}_g/\mc{M}_g}$, we know that there's a universal line bundle $\mathbb{L}^{\text{univ}}$ on $\mc{X}$. On each fiber of $p_2$, the line bundle $\mathbb{L}^{\text{univ}}$ restricts to a line bundle of degree $1$. We would like to modify $\mathbb{L}^{\text{univ}}$ so that the restriction is homologically trivial. Let $\omega:=p_1^*\omega_{\mc{C}_g/\mc{M}_g}$ where $\omega_{\mc{C}_g/\mc{M}_g}$ is the relative canonical bundle on the universal curve $\mc{C}_g\to \mc{M}_g$. Then we see that the line bundle $$\mathbb{L}:=(\mathbb{L}^{\text{univ}})^{2g-2}\otimes \omega^{\vee}$$ has the desired property, as the restriction of $\mathbb{L}$ to each fiber of $p_2$ is of degree $0$. Let $D$ be a divisor corresponding to $\mathbb{L}$ and let $|D|$ be the support of $D$ together with the inclusion map $\iota\colon |D|\to \mc{X}$.

Now the relative cycle class map yields an exact sequence $$0\to R^1(p_2)_*\Z(1)\to R^1(p_2|_{\mc{X}-|D|})_*\Z(1)\to (p_2|_{|D|})_*\mc{H}^2_{|D|}(\Z(1))\to R^2(p_2)_*\Z(1)\to \dots$$ where $\mc{H}^2_{|D|}(\Z(1)):=R^2\iota^!\Z(1)$ is the sheaf on $|D|$ whose pushforward gives the bundle of local cohomology groups with respect to the subset $|D|$ on $\mc{X}$. Since $D$ is homologically trivial on each fiber, we see that the image of the natural inclusion map $\Z\to (p_2|_{|D|})_*\mc{H}^2_{|D|}(\Z(1))$ induced by $D$ goes to $0$ in $R^2(p_2)_*\Z(1)$ and hence we may pullback to get a short exact sequence 
\[\begin{tikzcd}[sep=small]
	0 & {R^1(p_2)_*\Z(1)} & {R^1(p_2|_{\Pic^1_{\mc{C}_g/\mc{M}_g}-|D|})_*\Z(1)} & {(p_2|_{|D|})_*\mc{H}^{2}_{|D|}(\Z(1))} & {R^2(p_2)_*\Z(1)} \\
	0 & {R^1(p_2)_*\Z(1)} & {\mc{E}_0} & {\Z(0)} & 0
	\arrow[from=1-1, to=1-2]
	\arrow[from=1-2, to=1-3]
	\arrow[from=1-3, to=1-4]
	\arrow[from=1-4, to=1-5]
	\arrow[from=2-1, to=2-2]
	\arrow["{=}", from=2-2, to=1-2]
	\arrow[from=2-2, to=2-3]
	\arrow[from=2-3, to=1-3]
	\arrow[from=2-3, to=2-4]
	\arrow[from=2-4, to=1-4]
	\arrow[from=2-4, to=2-5]
	\arrow[from=2-5, to=1-5]
\end{tikzcd}\]

Notice that all the maps involved are maps of variations of mixed Hodge structure, so $\mc{E}_0$ is a local system which underlies a graded-polarizable, admissible variation of mixed Hodge structure as well. Furthermore, since the fibers of $p_2$ over any point of $\Pic^1(C)\subset \Pic^1_{\mc{C}_g/\mc{M}_g}$ is simply $C$, if we restrict $\mc{E}_0$ to $\Pic^1(C)$, we get a unipotent local system $$0\to \underline{H^1(C,\Z(1))}\to (\mc{E}_0)|_{\Pic^1(C)}\to \underline{\Z(0)}\to 0$$ We would like to compute the monodromy of this unipotent local system. 

Recall that monodromy is a topological invariant and does not depend on the base point we choose. Let $x\in C$. This corresponds to a degree $1$ line bundle on $C$ and hence gives us a point $[x]\in \Pic^1(C)$. We compute the monodromy with respect to this base point. Since the local system is unipotent, the monodromy reprensentation is determined by a map $$H_1(C,\Z)=\pi_1(\Pic^1(C),[x])\to \Hom(\Z(0),H^1(C,\Z(1)))=H^1(C,\Z(1))=H_1(C,\Z).$$ In particular, the monodromy of the local system $(\mc{E}_0)|_{\Pic^1(C)}$ is faithful if $\rho_{\mc{E}_0}$ is injective.

Observe that $[x]\in \Pic^1(C)$ is in the image of the natural inclusion map $C\to \Pic^1(C)$ and we have the following commutative diagram 
\[\begin{tikzcd}
	C & {\Pic^1(C)} \\
	{\mc{C}_g} & {\Pic^1_{\mc{C}_g/\mc{M}_g}}
	\arrow[hook, from=1-1, to=1-2]
	\arrow[hook, from=1-1, to=2-1]
	\arrow[hook, from=1-2, to=2-2]
	\arrow["j", hook, from=2-1, to=2-2]
\end{tikzcd}\] The inclusion from $C$ into $\Pic^1(C)$ induces a surjection of topological fundamental group $\pi_1(C,x)\twoheadrightarrow \pi_1(\Pic^1(C),[x])$ and hence to compute the monodromy of $(\mc{E}_0)|_{\Pic^1(C)}$, it's enough to compute the monodromy of $(j^*\mc{E}_0)|_{C}$.

Now if we view $\mc{C}_g$ as $\mc{M}_{g,1}$, when we pull the family $p_2\colon \mc{C}_g\times_{\mc{M}_g}\to \Pic^1_{\mc{C}_g/\mc{M}_g}$ back along this embedding $j$, we get $$(\mc{C}_g\times_{\mc{M}_g}\Pic^1_{\mc{C}_g/\mc{M}_g})\times_{\Pic^1_{\mc{C}_g/\mc{M}_g}}\mc{C}_{g,1}=\mc{M}_{g,1}\times_{\mc{M}_g}\mc{M}_{g,1}=\mc{C}_{g,1}\to \mc{M}_g$$ which is the universal pointed curve over $\mc{M}_{g,1}$, and the divisor can be chosen to be $(2g-2)\xi-K$, where $\xi$ is the canonical section from $\mc{M}_{g,1}$ to $\mc{C}_{g,1}$, and $K$ is a relative canonical divisor for the map $\mc{C}_{g,1}\to \mc{M}_{g,1}$. In particular, we can relate $j^*\mc{E}_0$ to a certain local system studied by Hain and Matsumoto in \cite{HM05}. Building on their monodromy computation \cite[Proposition 6.4]{HM05}, we can prove the following:
\begin{proposition}\label{Proposition: key monodromy computation}
	The monodromy representation is given by
\[\begin{tikzcd}
	{\pi_1(C,x)} & {\pi_1(\Pic^1(C),[x])} & {H_1(C,\Z)} \\
	\gamma && {(g-1)[\gamma]}
	\arrow[two heads, from=1-1, to=1-2]
	\arrow[from=1-2, to=1-3]
	\arrow[maps to, from=2-1, to=2-3]
\end{tikzcd}\] In particular, $(\mc{E}_0)|_{\Pic^1(C)}$ is faithful.
\end{proposition}
The proof of Proposition \ref{Proposition: key monodromy computation} goes through the Johnson homomorphism and since it requires us to go on a tangent, we defer the proof to the Appendix.

\begin{proof}[Proof of Theorem \ref{Theorem: injectivity for Pic^1}]
It remains to write down local systems $(\mc{E}_i)_{i\geq 1}$ whose restrictions to $\Pic^1(C)$ are of finite monodromy and are jointly faithful. Let $\rho\colon \pi_1(\Pic^1_{\mc{C}_g/\mc{M}_g})\to \GL_r(\C)$ be the monodromy representation associated to $\mc{E}_0$. Consider all finite quotients of the image of $\rho$. As every finite group is linear, we may pick faithful representations of these finite quotients and hence we get local systems $\mc{E}_i$ on $\Pic^1_{\mc{C}_g/\mc{M}_g}$ of finite monodromy, and hence underlie graded-polarizable, admissible $\Q$-variations of mixed Hodge structure on $\Pic^1_{\mc{C}_g/\mc{M}_g}$. Consider the restriction of $\mc{E}_i$ to the fiber $\Pic^1(C)$. The restrictions are certainly of finite monodromy, and it remains to show that they are jointly faithful. Since we showed that the restriction of $\rho$ to $\pi_1(\Pic^1(C))$ is faithful, we may view it as a subgroup of the image of $\rho$. Since $\rho(\pi_1(\Pic^1_{\mc{C}_g/\mc{M}_g}))$ is a finitely generated linear group, it's residually finite \cite{Mal40}, and hence for any $g\in \pi_1(\Pic^1(C))$, there's some finite quotient of $\rho(\pi_1(\Pic^1_{\mc{C}_g/\mc{M}_g}))$ for which $g$ remains non-zero. Hence, the $\mc{E}_i|_{\Pic^1(C)}$'s are jointly faithful as desired.
\end{proof}
\begin{remark}
	Note that since $\pi_1(\mc{M}_g)$ is residually finite \cite[Theorem 6.11]{FM12}, this argument also shows that the topological fundamental group $\pi_1(\Pic^1_{\mc{C}_g/\mc{M}_g})$ of $\Pic^1_{\mc{C}_g/\mc{M}_g}$ is residually finite. In fact, suppose $f\colon X\to B$ is any smooth projective family of curves over some base $B$ with $\pi_1(B)$ residually finite. Then the same argument shows that the topological fundamental group $\pi_1(\Pic^1_{X/B})$ is residually finite.
\end{remark}

Notice that in the example of $p\colon \Pic^1_{\mc{C}_g/\mc{M}_g}\to \mc{M}_g$, the group $\pi_1(\LS^{\text{Hdg}}(\Pic^1(C)))$ does not depend on the choice of base point $[C]$ in $\mc{M}_g$. Recall that in the abelian setting, the Hodge exceptional locus also has a group-theoretic characterization: it's the locus where the Mumford-Tate group of the fiber becomes smaller than the Mumford-Tate group of a general fiber. Therefore, it's reasonable to ask the following question 
\begin{question}
	Given a smooth projective families of algebraic varieties $f\colon X\to B$, can $NG(X/B)$ be characterized as the points $b$ of $B$ such that $\pi_1(\LS^\text{Hdg}(X_b))$ differs from that of a general fiber? 
\end{question}
If this question has an affirmative answer, then using Proposition \ref{Prop: Tannakian fundamental group for A.V.}, we can for example conclude that for any smooth projective family of abelian varieties $f\colon \mc{A}\to B$, the locus $NG(\mc{A}/B)$ is always empty.

\section{Hodge theoretic section question and relation to \'{e}tale fundamental groups}\label{section: obstruction to sections}
In this section, we continue the discussion of Hodge theoretic section question initiated in \cite{Xu24} (see also the thesis of Ferrario \cite{Fer20}). The main goals are twofold: we would like to modify the formulation of the Hodge theoretic section conjecture in \cite{Xu24}, and then we would like to explain in particular how \'{e}tale fundamental groups can obstruct existence of sections to exact sequences of Hodge theoretic fundamental groups. 

Let $f\colon X\to B$ be a smooth projective family of algebraic varieties. Any algebraic section $s\colon B\to X$ induces a splitting of the sequence $(\ref{Hodge SES})$. This induces a map $$\text{sec}\colon \{\text{algebraic sections to }f\colon X\to B\}\to \{\text{splittings of (\ref{Hodge SES})}\}/\backsim,$$ where the equivalence relation on the set of splittings is induced by isomorphism of functors via Tannakian duality. Inspired by Grothendieck's section conjecture, we can formulate the Hodge theoretic section question, which asks when the map $\text{sec}$ is a bijection.

Now in the anabelian philosophy (see \cite{Fal98} for an overview), group theoretic maps are not simply maps of \'{e}tale fundamental groups, but rather maps of short exact sequences. More precisely, given a smooth projective variety $X$ over some number field $k$, we have the associated sequence of \'{e}tale fundamental groups $$1\to \pi_1^{\text{\'{e}t}}(X_{\overline{k}})\to \pi_1^{\text{\'{e}t}}(X)\to \Gal(\overline{k}/k)\to 1$$ where $X_{\overline{k}}$ is the base change of $X$ to a fixed separable closure $\overline{k}$ over $k$, and the group theoretic maps in anabelian geometry really are maps between this fundamental short exact sequences up to conjugation by the geometric fundamental group.

Going back to the Hodge theoretic world, we may view the short exact sequence proved by Andr\'{e} \cite[Theorem C. 11]{And21} as a Hodge theoretic analogue of the fundamental exact sequence of \'{e}tale fundamental groups (see also \cite[Introduction]{Ara10}). This motivates the following definition:
\begin{definition}
	Suppose $X$ and $Y$ are two smooth connected complex algebraic varieties. Then a map $f\colon \pi_1(\VMHS(X))\to \pi_1(\VMHS(Y))$ is called \textit{Hodge theoretic} if it induces a map of short exact sequences 
\[\begin{tikzcd}
	1 & {\pi_1(\LS^{\text{Hdg}}(X))} & {\pi_1(\VMHS(X))} & {\pi_1(\MHS)} & 1 \\
	1 & {\pi_1(\LS^{\text{Hdg}}(Y))} & {\pi_1(\VMHS(Y))} & {\pi_1(\MHS)} & 1
	\arrow[from=1-1, to=1-2]
	\arrow[from=1-2, to=1-3]
	\arrow[from=1-2, to=2-2]
	\arrow[from=1-3, to=1-4]
	\arrow[from=1-3, to=2-3]
	\arrow[from=1-4, to=1-5]
	\arrow[from=1-4, to=2-4]
	\arrow[from=2-1, to=2-2]
	\arrow[from=2-2, to=2-3]
	\arrow[from=2-3, to=2-4]
	\arrow[from=2-4, to=2-5]
\end{tikzcd}\]
\end{definition}
Note that any algebraic morphism between $X$ and $Y$ induces Hodge theoretic maps between their Hodge theoretic fundamental groups. Then we may modify the definition of the section map: instead of going into the set of group theoretic splittings of (\ref{Hodge SES}), we asks that it lands in the set of Hodge theoretic splittings:$$\sec_{\text{Hdg}}\colon \{\text{algebraic sections to }f\colon X\to B\}\to \{\text{Hodge theoretic splittings of }(\ref{Hodge SES})\}/\backsim,$$ where the equivalence relation is given by conjugation by $\pi_1(\LS^{\text{Hdg}}(Y))$. 
\begin{question}[Hodge theoretic section question]
	When is $\text{sec}_{\text{Hdg}}$ a bijection?
\end{question}
Perhaps the simplest way of verifying that the map $\text{sec}_{\text{Hdg}}$ is a bijection is to show that the sequence (\ref{Hodge SES}) does not admit any Hodge theoretic splitting. We show that sometimes it's enough to check at the level of \'{e}tale fundamental groups:
\begin{theorem}\label{Theorem: Hodge theoretic v.s. etale}
	Let $f\colon X\to B$ be a smooth projective map between smooth connected complex \textit{quasi-projective} varieties. If $\pi_1(\VMHS(X))\to \pi_1(\VMHS(B))$ admits a Hodge theoretic splitting, then the induced map of \'{e}tale fundamental groups $\pi_1^{\text{\'{e}t}}(X)\to \pi_1^{\text{\'{e}t}}(B)$ also splits.
\end{theorem}
\begin{proof}
	Suppose that we have a Hodge theoretic section $s\colon \pi_1(\VMHS(B))\to \pi_1(\VMHS(X))$. By definition, $s$ restricts to a splitting $s|_{\text{LS}}\colon \pi_1(\LS^{\text{Hdg}}(B))\to \pi_1(\LS^{\text{Hdg}}(X))$ of the map $\pi_1(\LS^{\text{Hdg}}(X))\to \pi_1(\LS^{\text{Hdg}}(B))$. Applying Tannakian duality, we see that at the categorical side, we get a functor $s^*|_{\text{LS}}\colon \LS^{\text{Hdg}}(X)\to \LS^{\text{Hdg}}(B)$ such that $s^*|_{\text{LS}}\circ f^*$ is isomorphic to the identity functor on $\LS^{\text{Hdg}}(B)$. 
	
	We claim that any additive tensor functor between $\LS^{\text{Hdg}}(X)$ and $\LS^\text{Hdg}(B)$ has to preserve local systems of finite monodromy. Assuming this claim is true, then $s^*|_{\text{LS}}$ further restricts to a functor from the category $\LS_{\text{fin}}(X)$ of local systems with finite monodromy on $X$ to $\LS_{\text{fin}}(B)$, such that precomposing with $f^*$ is isomorphic to the identity functor on $\LS_{\text{fin}}(B)$. This concludes the proof as the \'{e}tale fundamental group of a smooth connected complex algebraic variety $X$ can be identified with the Tannakian fundamental group $\pi_1(\LS_{\text{fin}}(X))$ (see \cite[Example 2.34]{DM82}).
\end{proof}
Therefore, it remains to verify the following lemma:
\begin{lemma}
	Let $X$ and $Y$ be two smooth connected complex quasi-projective varieties. Then any additive tensor functors between (sub)categories of local systems have to send local systems of finite monodromy to local systems of finite monodromy.
\end{lemma}
\begin{proof}
	This follows from Nori's criterion for a local system to be of finite monodromy: a local system $\mathbb{V}$ on a smooth connected quasi-projective variety $X$ is of finite monodromy if and only if there are two polynomials $P(x)\neq Q(x)\in \mathbb{N}[x]$ such that $P(\mathbb{V})=Q(\mathbb{V})$, where multiplication of local system is interpreted as tensor products and addition is interpreted as direct sum. This is first proven by Nori in the case where $X$ is a curve \cite[Proposition 3.10]{Nor76}, and is later generalized to the case where $X$ is a smooth connected quasi-projective variety by the work of Biswas-Holla-Schumacher \cite[Theorem 3.3]{BHS00}.
\end{proof}
\begin{remark}
	This criterion also works for $\mc{M}_g$ and $\Pic^1_{\mc{C}_g/\mc{M}_g}$ as one can check if a local system is of finite monodromy by pulling back this local system to a finite \'{e}tale cover, and both $\mc{M}_g$ and $\Pic^1_{\mc{C}_g/\mc{M}_g}$ admit finite \'{e}tale covering maps by smooth connected quasi-projective varieties.
\end{remark}
\begin{corollary}\label{Corollary: Hodge theoretic section question for M_g}
	Fix some $g\geq 3$. Let $U$ be any open subset of $\mc{M}_g$ and $f^{-1}(U)$ its preimage in $\mc{C}_g$ and $p^{-1}(U)$ its preimage in $\Pic^1_{\mc{C}_g/\mc{M}_g}$. Then the maps $$\pi_1(\VMHS(p^{-1}(U)))\to \pi_1(\VMHS(U)),\quad\quad\pi_1(\VMHS(f^{-1}(U)))\to \pi_1(\VMHS(U))$$ do not admit any Hodge theoretic sections, and the section maps $\text{sec}_{\text{Hdg}}$ are bijections in these two cases. In particular, taking $U$ to be $\mc{M}_g$ recovers the second part of Theorem \ref{Main theorem}.
\end{corollary}
\begin{proof}
	This follows from the fact that the associated maps of \'{e}tale fundamental groups do not split when the base is the generic point of $\mc{M}_g$ (see \cite[Theorem 2]{Hai11} and \cite[Corollary 1.3.3. and Corollary 1.3.4]{LLSS}).
\end{proof}
\begin{remark}
	Andr\'{e} observed in \cite[Corollary C.14]{And21} that given a smooth connected complex algebraic variety $X$, the \'{e}tale fundamental group $\pi_1^{\text{\'{e}t}}(X)$ can be identified with $\pi_0(\pi_1(\VMHS(X)))=\pi_0(\pi_1(\LS^{\text{Hdg}}(X)))$. In particular, Theorem \ref{Theorem: Hodge theoretic v.s. etale} still holds even if we don't impose the condition that the section is Hodge theoretic. It's then natural to ask if every group theoretic section is automatically Hodge theoretic. While we do not believe it's the case, we don't currently have a counterexample.
\end{remark}

\section{Appendix: Proof of Proposition \ref{Proposition: key monodromy computation}}
All the ingredients needed for a proof of Proposition \ref{Proposition: key monodromy computation} are contained in \cite{HM05}, although they did not do this monodromy computation explicitly. They were mostly concerned with computing the monodromy a related local system, which we will call $\mc{E}_\text{cer}$, on $\mc{M}_{g,1}$ and we explain how to compute the monodromy of $j^*\mc{E}_0$ by using their results and some standard facts about the Johnson homomorphism (see \cite[section 6.6]{FM12}).

We first set up some notations. Fix some integer $g\geq 2$. Let $f\colon \Pic^0_{\mc{C}_{g,1}/\mc{M}_{g,1}}\to \mc{M}_{g,1}$ be the relative Jacobian of the universal curve $\mc{C}_{g,1}$ of genus $g$ with $1$ marked point and $[C,x]\in \mc{M}_{g,1}$ some base point. Let $\xi\colon \mc{M}_{g,1}\to \mc{C}_{g,1}$ be the canonical section. Define two local systems on $\mc{M}_{g,1}$: $$\mathbb{L}_\Z:=R^{2g-3}f_*\Z(g-1)\quad\quad \mathbb{H}_{\Z}:=R^{2g-1}f_*\Z(g).$$ Note that $\mathbb{H}_\Z$ is the local system corresponding to the action of $\pi_1(\mc{M}_{g,1},[C,x])$ on $H:=H^1(C,\Z(1))$ and $\mathbb{L}_\Z$ is the local system corresponding to the action of $\pi_1(\mc{M}_{g,1},[C,x])$ on $L:=[\wedge^3H](-1)$. There's a natural projection map $c\colon L\to H$ defined by $$
	c(x\wedge y\wedge z)=	\omega(x,y)z+\omega(y,z)x+\omega(z,x)y
$$ where $\omega$ is the intersection form. 

The precise definition of the local system $\mc{E}_{\text{cer}}$ is not relevant to our story. The key is that it's an extension of the form $$0\to \mathbb{L}_\Z\to \mc{E}_{\text{cer}}\to \underline{\Z(0)}\to 0.$$ More importantly, by \cite[Corollary 6.7]{Pul88}, we know that the map $c$ is equivariant with respect to the action of $\pi_1(\mc{M}_{g,1})$ and hence induces a commutative diagram: 
\[\begin{tikzcd}
	0 & {\mathbb{L}_\Z} & {\mc{E}_{\text{cer}}} & {\underline{\Z(0)}} & 0 \\
	0 & {\mathbb{H}_\Z} & {j^*\mc{E}_0} & {\underline{\Z(0)}} & 0
	\arrow[from=1-1, to=1-2]
	\arrow[from=1-2, to=1-3]
	\arrow[two heads, from=1-2, to=2-2]
	\arrow[from=1-3, to=1-4]
	\arrow[from=1-3, to=2-3]
	\arrow[from=1-4, to=1-5]
	\arrow[from=1-4, to=2-4]
	\arrow[from=2-1, to=2-2]
	\arrow[from=2-2, to=2-3]
	\arrow[from=2-3, to=2-4]
	\arrow[from=2-4, to=2-5]
\end{tikzcd}\] In particular, the monodromy representation associated to $(j^*\mc{E}_0)|_C$ is determined by sending a loop $\gamma\in \pi_1(C,x)$ to a map of the form $$\Z\to L\xrightarrow{c} H.$$ Fix such a loop $\gamma$. Notice that $\pi_1(C,x)$ must be contained in the Torelli subgroup $T_{g,1}$ of $\pi_1(\mc{M}_{g,1})$ as it's already trivial under the projection onto $\pi_1(\mc{M}_g)$. The key result of Hain and Matsumoto is the following 
\begin{proposition}[Proposition 6.4 in \cite{HM05}]
	The monodromy of $\mc{E}_{\text{cer}}$ around $\gamma$ is given by the map $\gamma\mapsto \tau([\gamma])$, where $\tau$ is the Johnson homomorphism $\tau\colon T_{g,1}\to L$, and $[\gamma]$ is the homology class of $\gamma$ in $H_1(C,\Z)$.
\end{proposition}
Proposition \ref{Proposition: key monodromy computation} now follows from a direct computation: 
\begin{proof}[Proof of Proposition \ref{Proposition: key monodromy computation}]
	It's enough to show that the composition $$\pi_1(C,x)\to T_{g,1}\xrightarrow{\tau}L\xrightarrow{c}H$$ is given by sending $\gamma$ to $(g-1)\cdot [\gamma]$. Now we know that the map $\pi_1(C,x)\to T_{g,1}\xrightarrow{\tau} L$ has the following explicit description (see e.g. \cite{FM12}): let $\{x_i, y_i\}_{i=1}^g$ be a symplectic basis for $H$ with respect to the intersection form $\omega$ and then this map is given by 
$$\begin{aligned}
	\pi_1(C,x)&\longrightarrow L\\
	\gamma &\mapsto\left(\sum x_i\wedge y_i\right)\wedge [\gamma]
\end{aligned}$$
Suppose $[\gamma]=\sum \alpha_i x_i+\beta_j y_i$. We may now compute that $$\begin{aligned}c\left(\left(\sum x_i\wedge y_i \right)\wedge [\gamma]\right)&=\sum_{i=1}^g c(x_i\wedge y_i\wedge [\gamma])\\&=\sum_{i=1}^g[\gamma]+\omega(y_i,[\gamma])x_i+\omega([\gamma],x_i)y_i\\&=\sum_{i=1}^g[\gamma]-\alpha_i x_i-\beta_i y_i\\&=(g-1)[\gamma]
\end{aligned}
$$ This concludes the proof of Proposition \ref{Proposition: key monodromy computation}.
\end{proof}

\bibliographystyle{alpha}
\bibliography{bibliography-hodgefundamentalgroup.bib}

\begin{thebibliography}{LLSS23}

\bibitem[ACG11]{ACG11}
Enrico Arbarello, Maurizio Cornalba, and Phillip~A. Griffiths.
\newblock {\em Geometry of algebraic curves. {V}olume {II}}, volume 268 of {\em Grundlehren der mathematischen Wissenschaften [Fundamental Principles of Mathematical Sciences]}.
\newblock Springer, Heidelberg, 2011.
\newblock With a contribution by Joseph Daniel Harris.

\bibitem[And21]{And21}
Yves Andr\'{e}.
\newblock Normality criteria, and monodromy of variations of mixed {H}odge structures, {A}ppendix {C} to {B}runo {K}ahn's {A}lbanese kernels and {G}riffiths group.
\newblock {\em Tunis. J. Math.}, 3(3):589--656, 2021.

\bibitem[Ara10]{Ara10}
Donu Arapura.
\newblock The {H}odge theoretic fundamental group and its cohomology.
\newblock In {\em The geometry of algebraic cycles}, volume~9 of {\em Clay Math. Proc.}, pages 3--22. Amer. Math. Soc., Providence, RI, 2010.

\bibitem[BHS00]{BHS00}
Indranil Biswas, Yogish~I. Holla, and Georg Schumacher.
\newblock On a characterization of finite vector bundles as vector bundles admitting a flat connection with finite monodromy group.
\newblock {\em Proc. Amer. Math. Soc.}, 128(12):3661--3669, 2000.

\bibitem[CDK95]{CDK95}
Eduardo Cattani, Pierre Deligne, and Aroldo Kaplan.
\newblock On the locus of {H}odge classes.
\newblock {\em J. Amer. Math. Soc.}, 8(2):483--506, 1995.

\bibitem[DE22]{DE22}
Marco D'Addezio and H\'{e}l{\`e}ne Esnault.
\newblock On the universal extensions in {T}annakian categories.
\newblock {\em Int. Math. Res. Not. IMRN}, (18):14008--14033, 2022.

\bibitem[Del71]{Del71}
Pierre Deligne.
\newblock Th\'{e}orie de {H}odge. {II}.
\newblock {\em Inst. Hautes \'{E}tudes Sci. Publ. Math.}, (40):5--57, 1971.

\bibitem[Del87]{Del87}
P.~Deligne.
\newblock Un th\'{e}or{\`e}me de finitude pour la monodromie.
\newblock In {\em Discrete groups in geometry and analysis ({N}ew {H}aven, {C}onn., 1984)}, volume~67 of {\em Progr. Math.}, pages 1--19. Birkh\"{a}user Boston, Boston, MA, 1987.

\bibitem[DM18]{DM82}
P.~Deligne and J.~S. Milne.
\newblock Tannakian categories.
\newblock https://www.jmilne.org/math/xnotes/tc2018.pdf, 2018.

\bibitem[EK24]{EK24}
H{\'e}l{\`e}ne Esnault and Moritz Kerz.
\newblock A non-abelian version of {D}eligne's {F}ixed {P}art {T}heorem.
\newblock {\em arXiv: 2408.13910v4}, 2024.

\bibitem[Fal98]{Fal98}
Gerd Faltings.
\newblock Curves and their fundamental groups (following {G}rothendieck, {T}amagawa and {M}ochizuki).
\newblock Number 252, pages 131--150. 1998.
\newblock S\'{e}minaire Bourbaki. Vol. 1997/98.

\bibitem[Fer20]{Fer20}
Riccardo Ferrario.
\newblock {\em A complex analogue of Grothendieck's section conjecture}.
\newblock PhD thesis, ETH, 2020.

\bibitem[FM12]{FM12}
Benson Farb and Dan Margalit.
\newblock {\em A primer on mapping class groups}, volume~49 of {\em Princeton Mathematical Series}.
\newblock Princeton University Press, Princeton, NJ, 2012.

\bibitem[Hai11]{Hai11}
Richard Hain.
\newblock Rational points of universal curves.
\newblock {\em J. Amer. Math. Soc.}, 24(3):709--769, 2011.

\bibitem[Hit87]{Hit87}
N.~J. Hitchin.
\newblock The self-duality equations on a {R}iemann surface.
\newblock {\em Proc. London Math. Soc. (3)}, 55(1):59--126, 1987.

\bibitem[HM05]{HM05}
Richard Hain and Makoto Matsumoto.
\newblock Galois actions on fundamental groups of curves and the cycle {$C-C^-$}.
\newblock {\em J. Inst. Math. Jussieu}, 4(3):363--403, 2005.

\bibitem[HZ87]{HZ87}
Richard~M. Hain and Steven Zucker.
\newblock Unipotent variations of mixed {H}odge structure.
\newblock {\em Invent. Math.}, 88(1):83--124, 1987.

\bibitem[Jac22]{Jac22}
Emil Jacobsen.
\newblock Malcev completions, hodge theory, and motives.
\newblock {\em arXiv: 2205.13073v5}, 2022.

\bibitem[Kat72]{Kat72}
Nicholas~M. Katz.
\newblock Algebraic solutions of differential equations ({$p$}-curvature and the {H}odge filtration).
\newblock {\em Invent. Math.}, 18:1--118, 1972.

\bibitem[Lam22]{Lam22}
Yeuk Hay~Joshua Lam.
\newblock Motivic local systems on curves and {M}aeda's conjecture.
\newblock {\em arxiv: 2211.06120v1}, 2022.

\bibitem[LL24]{LL24}
Aaron Landesman and Daniel Litt.
\newblock Geometric local systems on very general curves and isomonodromy.
\newblock {\em J. Amer. Math. Soc.}, 37(3):683--729, 2024.

\bibitem[LLSS23]{LLSS}
Wanlin Li, Daniel Litt, Nick Salter, and Padmavathi Srinivasan.
\newblock Surface bundles and the section conjecture.
\newblock {\em Math. Ann.}, 386(1-2):877--942, 2023.

\bibitem[Mal40]{Mal40}
A.~Malcev.
\newblock On isomorphic matrix representations of infinite groups.
\newblock {\em Rec. Math. [Mat. Sbornik] N.S.}, 8(50):405--422, 1940.

\bibitem[Mil79]{Mil79}
John~J. Millson.
\newblock Real vector bundles with discrete structure group.
\newblock {\em Topology}, 18(1):83--89, 1979.

\bibitem[Nor76]{Nor76}
Madhav~V. Nori.
\newblock On the representations of the fundamental group.
\newblock {\em Compositio Math.}, 33(1):29--41, 1976.

\bibitem[Pul88]{Pul88}
Michael~J. Pulte.
\newblock The fundamental group of a {R}iemann surface: mixed {H}odge structures and algebraic cycles.
\newblock {\em Duke Math. J.}, 57(3):721--760, 1988.

\bibitem[Sim92]{Sim92}
Carlos~T. Simpson.
\newblock Higgs bundles and local systems.
\newblock {\em Inst. Hautes \'{E}tudes Sci. Publ. Math.}, (75):5--95, 1992.

\bibitem[Sim97]{Sim97}
Carlos Simpson.
\newblock The {H}odge filtration on nonabelian cohomology.
\newblock In {\em Algebraic geometry---{S}anta {C}ruz 1995}, volume~62 of {\em Proc. Sympos. Pure Math.}, pages 217--281. Amer. Math. Soc., Providence, RI, 1997.

\bibitem[SZ85]{SZ85}
Joseph Steenbrink and Steven Zucker.
\newblock Variation of mixed {H}odge structure. {I}.
\newblock {\em Invent. Math.}, 80(3):489--542, 1985.

\bibitem[Tak83]{Tak83}
Kisao Takeuchi.
\newblock Arithmetic {F}uchsian groups with signature {$(1;e)$}.
\newblock {\em J. Math. Soc. Japan}, 35(3):381--407, 1983.

\bibitem[Voi13]{Voi13}
Claire Voisin.
\newblock Hodge loci.
\newblock In {\em Handbook of moduli. {V}ol. {III}}, volume~26 of {\em Adv. Lect. Math. (ALM)}, pages 507--546. Int. Press, Somerville, MA, 2013.

\bibitem[Xu24]{Xu24}
Simon~Shuofeng Xu.
\newblock Section conjectures over $\mathbb{C}$ and {K}odaira fibrations.
\newblock {\em arxiv: 2407.03248v2}, 2024.

\end{thebibliography}

\end{document}